\documentclass[12pt]{amsart}
\usepackage{amssymb,latexsym}
\usepackage{enumerate}
\usepackage[margin=2in]{geometry}

\makeatletter
\@namedef{subjclassname@2010}{%
  \textup{2010} Mathematics Subject Classification}
\makeatother

\newtheorem{thm}{Theorem}
\newtheorem{cor}{Corollary}
\newtheorem{prop}{Proposition}

\theoremstyle{definition}

\newtheorem*{xrem}{Remark}

\numberwithin{equation}{section}

\newcommand\R{\mathbb{R}}

\newcommand{\sr}[1]{\mathfrak{M}_{#1}}
\newcommand{\srpot}[1]{\mathfrak{N}_{#1}}
\newcommand{\norma}[1]{\left\| #1 \right\| }

\newcommand{\constproved}{4}
\newcommand{\hamymean}[1]{\mathfrak{ha}^{[#1]}}
\newcommand{\hayashimean}[1]{\mathfrak{hy}^{[#1]}}

\begin{document}


\baselineskip=17pt



\title[On some type of Hardy's inequality]{On some type of Hardy's inequality involving generalized power means}

\author[P. Pasteczka]{Pawe{\l} Pasteczka}
\address{Institute of Mathematics\\University of Warsaw\\
02-097 Warszawa, Poland}
\email{ppasteczka@mimuw.edu.pl}

\begin{abstract}
We discuss properties of certain generalization of Power Means proposed in 1971 by Carlson, Meany and Nelson. For any fixed parameter $(k,s,q)$ and vector $(v_1,\ldots,v_n)$ they take the $q$-th power means of all possible $k$-tuples $(v_{i_1},\ldots,v_{i_k})$, and then calculate the $s$-th power mean of the resulting vector of length $C_{n}^{k}$.

We work towards a complete answer to the question when such means satisfy inequalities resembling the classical Hardy inequality. We give a definitive answer in a large part of the parameter space.
 
An unexpected corollary is that this family behaves much differently than most of other families admitting Hardy-type inequalities. Namely, arbitrarily small perturbations of parameters \textit{may} lead to the breakdown of such inequalities.
\end{abstract}



\maketitle

\section{Introduction}

The most popular family of means encountered in the literature consists of Power Means.
They are parametrized by a free parameter $p \in \R\cup \{ \pm \infty\}$ and, for any 
all-positive-components vector $(v_1,\ldots,v_n)$, are defined by
$$
\srpot{p}(v_1,\dots,v_n):=
\begin{cases} 
\left(\frac{1}{n} \sum_{i=1}^{n} v_i^p\right)^{1/p}	& \textrm{if\ } p \ne 0\,,\\
\left(\prod_{i=1}^{n} v_i\right)^{1/n}	& \textrm{if\ } p = 0\,,\\
\min(v_1,\ldots,v_n)	& \textrm{if\ } p = - \infty\,,\\
\max(v_1,\ldots,v_n)	& \textrm{if\ } p = \infty\,.
\end{cases}
$$

Upon putting $p=1,\,0,\,-1$ one gets Arithmetic, Geometric and Harmonic Mean respectively.

Power Means have been investigated ever since their conception in the 19th century. 
One of the classical results concerning them is the following
inequality, in its final form due to Landau, \cite{Landau21}
\begin{equation}
\sum_{n=1}^{\infty} \srpot{p}(a_1,\dots,a_n) < (1-p)^{-1/p} \norma{a}_1 
\textrm{ for }p \in (0,1) \textrm{ and }a \in l_1(\R_{+})\,. \nonumber
\end{equation}
The constant $(1-p)^{-1/p}$ cannot be diminished for any $p$. 
Similar inequality (with a non-optimal constant, however) was proved one year earlier by Hardy, 
\cite{Hardy20}, in the course of his commenting some still earlier results of Hilbert.

In particular, for $p=\frac{1}{2}$ one has
\begin{equation}
\sum_{n=1}^{\infty} \srpot{1/2}(a_1,\dots,a_n) < 4 \norma{a}_1. \label{eq:1/2}
\end{equation}

Nowadays these results are few of many in the emerging ``theory of Hardy means''.
Precisely, one says that $\mathfrak{A}$ is a Hardy mean when there exists a constant $C$ such that
$$
\sum_{n=1}^{\infty} \mathfrak{A}(a_1,\dots,a_n) < C\norma{a}_1
\textrm{ for any }a \in l_1(\R_{+})\,.
$$
This definition was formally introduced by Pales and Persson in \cite{PaPe04}, but it had been felt in the air since the 1920s. The authors of \cite{PaPe04} proposed certain sufficient conditions, for a number of families of means, to be Hardy. Hence it is natural to ask what other means are Hardy? Indeed, this question was extensively dealt with decades before the formal definition appeared. The detailed history of the events related to, or implied by above inequalities is sketched in extensive and catching surveys \cite{PS,DMcG,OP}, and in a recent book \cite{KMP}. 

Unfortunately, for many families of means the problem if they are Hardy remains open. Such a problem for the two-parameter family of Gini means was, for instance, considered in \cite{PaPe04}, where many cases were solved. In fact, this open problem was explicitly worded three years later in \cite[p. 89]{KMP}. \\
Other interesting families which might be considered in this context are Hamy means (cf. \cite{Hamy}, \cite[p. 364]{bullen} and Corollary~\ref{cor:HamyMeans} below) and Hayashi means (cf. \cite{Hayashi}, \cite[p. 365]{bullen} and Corollary~\ref{cor:HayashiMeans}).  These two means are closely related with the present note.

We are now going to analyse a yet another multi-parameter family of means. Namely, in 1971 Carlson, Meany and Nelson \cite{CMN}, among other things, proposed the following family, witch encompasses Power Means, 
Hamy means and Hayashi means:
$$\sr{k,s,q}(v_1\ldots v_n):= 
\begin{cases}
\srpot{s}\Big(\srpot{q}(v_{i_1},\ldots,v_{i_k}) \colon i \in C_{n}^{k}\Big) & \textrm{if\ } k<n\,, \\ 
\srpot{q} (v_1,\ldots,v_n)& \textrm{if\ } k \geq n\,,
\end{cases}
$$
where $C_{n}^{k}:=\{A \subset \{1,\ldots n\} \colon \#A=k\}$ (we have slightly abused the notation of $\srpot{s}$). Those authors where interested in certain inequalities binding the means $\sr{k,s,q}$ when the order of parameters $s$ and $q$ is being reversed. \\
These very means turn out to be extremely useful for our purposes. Namely, we are going to prove that these means are Hardy for in a large part of the parameter space.

It is indeed a generalization of the above mention means.
Moreover, by \cite{sadikova}, there holds the following inequalities
\begin{align}
\sr{k,s,q} &\leq \sr{k,t,p}\textrm{ for }q \leq p\textrm{ and }s\leq t\,, \label{ineq:qs} \\
\sr{k,s,q} &\leq \sr{k-1,s,q}\textrm{ for }s>q\,. \label{ineq:k}
\end{align}

\section{Main Result}
In our main Theorem~\ref{thm:main} we are going to prove that $\sr{2,1,0}$ is a Hardy mean.
Obviously, all means majorized by some Hardy mean (or, more general, majorized up to some constant coefficient) are Hardy, too. Therefore, knowing that $\sr{2,1,0}$ is Hardy, the inequalities \eqref{ineq:k} and \eqref{ineq:qs} imply that $\sr{k,s,q}$ are Hardy, too, for a vast family of parameters. This is precisely worded in Proposition~\ref{prop:major} below.  

\begin{thm}
\label{thm:main}
$\sr{2,1,0}$ is a Hardy mean and 
$$\sum_{n=1}^{\infty} \sr{2,1,0} (a_1,\ldots,a_n) < 
\constproved \norma{a}_1 \textrm{ for every }a \in l^1(\mathbb{R}_{+})\,.$$
\end{thm}
\begin{proof}
We prove that $\sr{2,1,0}$ is majorized by $\srpot{1/2}$. Indeed, 
\begin{align}
\sr{2,1,0} (a_1,\ldots,a_n)&= {n \choose 2}^{-1} \sum_{1\leq i<j \leq n} \sqrt{a_ia_j} \nonumber \\
&= \tfrac{1}{n(n-1)} \sum_{1\leq i<j \leq n} 2\sqrt{a_ia_j} \nonumber \\
&= \tfrac{n}{n-1} \left(\left( \tfrac{1}{n}\sum_{i=1}^{n} \sqrt{a_i} \right)^2 - \frac{1}{n^2}\sum_{i=1}^{n} a_i \right)  \nonumber \\ 
&= \tfrac{n}{n-1} \left(\srpot{1/2}(a_1,\ldots,a_n)-\tfrac{1}{n}\srpot{1}(a_1,\ldots,a_n)\right) \nonumber \\
&\leq \tfrac{n}{n-1} \left(\srpot{1/2}(a_1,\ldots,a_n)-\tfrac{1}{n}\srpot{1/2}(a_1,\ldots,a_n)\right) \nonumber \\
&= \srpot{1/2}(a_1,\ldots,a_n)\,. \label{ineq:1/2}
\end{align}

 Hence, by \eqref{eq:1/2}, one obtains
$$\sum_{n=1}^{\infty} \sr{2,1,0} (a_1,\ldots,a_n) 
\leq \sum_{n=1}^{\infty} \srpot{1/2} (a_1,\ldots,a_n) 
< \constproved \norma{a}_1.
$$
\end{proof}

\begin{xrem}
The constant $\constproved$ in the above theorem cannot be diminished.
Indeed, upon taking $a_n=\tfrac{1}{n}$, a simple calculation yields
$$\lim_{n \rightarrow \infty} a_n^{-1} \sr{2,1,0} (a_1,\ldots,a_n)=\constproved.$$
Then the machinery originally put forward for the Power Means in \cite[pp. 241--242]{HLP}
is applicable here. We thus obtain that for an arbitrary $\varepsilon>0$ the sequence 
$$(a_1,\ldots,a_N,(N+1)^{-2},(N+2)^{-2},(N+3)^{-2},\ldots)$$
contradicts, for sufficiently large $N$, the constant being less then $\constproved-\varepsilon$.
\end{xrem}

\section{Discussion of parameters}

We know that for all $q_1>0$ and $C(k,q_1)=n^{-1/q_1}$ the inequality 
$$
\sr{q_1}(v_1,\ldots,v_k)>C(k,q_1)\max(v_1,\ldots,v_k)
$$
holds for any $v \in \R_{+}^k$. In particular for any $q_2 \in \mathbb{R} \cup \{\pm \infty\}$
\begin{equation}
\sr{q_1}(v_1,\ldots,v_k)>C(k,q_1)\sr{q_2}(v_1,\ldots,v_k)\,. \label{eq:bul_237}
\end{equation}

The above inequalities are instrumental in proving the following
\begin{prop}
\label{prop:major}
For any $k\ge2$ 
\begin{itemize}
\item $\sr{k,s,q}$ is a Hardy mean for no $s \ge 1$ and $q>0$,
\item $\sr{k,1,q}$ is a Hardy mean for any $q\le0$,
\item $\sr{k,s,q}$ is a Hardy mean for any $s<1$ and $q \in \mathbb{R} \cup \{\pm \infty\}$.
\end{itemize}
\end{prop}

\begin{proof}
Let us recall that the length of vectors in $\srpot{q}$ in the definition of $\sr{k,s,q}$ is fixed
(and equal $k$). It makes applying \eqref{eq:bul_237} very natural for the comparison of means 
with different $q$-s.

First item follows from the fact that $\sr{k,1,1}$ is an arithmetic mean. So it is not Hardy.
But $\sr{k,s,q}\geq\sr{k,1,q}\geq C\sr{k,1,1}$ for some constant $C$, hence $\sr{k,s,q}$ is not Hardy, too.

Second item is an immediate corollary from Theorem~\ref{thm:main} and \eqref{ineq:qs}.

Third item is being prove in two steps. First, with no loss of generality we may assume that $s\in(0,1)$.
We know that $\sr{k,s,s}=\srpot{s}$, so it is a Hardy mean.
Second, applying \eqref{eq:bul_237} one gets, $\sr{k,s,s} > C(k,s)\sr{k,s,q}$. So $\sr{k,s,q}$ is Hardy as well.
\end{proof}

\begin{cor}
\label{cor:HamyMeans}
For any $k\ge 2$ the Hamy mean $\hamymean{k}:=\sr{k,1,0}$ (cf. \cite{Hamy}, \cite[pp. 364--365]{bullen} for more details) is a Hardy mean.
\end{cor}

\begin{cor}
\label{cor:HayashiMeans}
For any $k\ge 2$ the Hayashi mean $\hayashimean{k}:=\sr{k,0,1}$ (cf. \cite{Hayashi}, \cite[pp. 365--366]{bullen} for more details) is a Hardy mean.
\end{cor}

The problem whether $\sr{k,s,q}$ is a Hardy mean for $q\leq0$, $s>1$ and $k \geq 2$ remains open.
Let us note that the possible answer may depend on $k$.

\end{document}